\documentclass[12pt, reqno]{amsart}
\usepackage{}
\usepackage{amsfonts}
\usepackage{bbm}
\usepackage{amscd,amsfonts}
\usepackage{amssymb, eucal, amsfonts, amsmath, xypic, latexsym}
\usepackage{pifont}
\usepackage{mathrsfs,color}
\usepackage{amsthm,indentfirst,bm,fancyhdr,dsfont}
\usepackage{graphicx}
\usepackage[all]{xy}
\usepackage[CJKbookmarks=true]{hyperref}

\newtheorem{theorem}{Theorem}[section]
\newtheorem{lemma}[theorem]{Lemma}
\newtheorem{Dynkin grading properties}[theorem]{Dynkin grading properties}
\newtheorem{prop}[theorem]{Proposition}

\theoremstyle{definition}

\newtheorem{example}[theorem]{Example}
\newtheorem{rem}[theorem]{Remark}

\def\ppp{\mathfrak{p}}
\def\ggg{\mathfrak{g}}
\def\gl{\mathfrak{gl}}
\def\co{\mathcal{O}}

\def\bbz{\mathbb{Z}}
\def\bbc{\mathbb{C}}
\def\bbk{\mathds{k}}

\def\dsk{\mathds{k}}

\def\col{\text{col}}
\def\row{\text{row}}

\def\bo{{\bar 1}}
\def\bz{{\bar 0}}

\def\deg{\text{deg}}
\def\ev{\text{ev}}

\numberwithin{equation}{section}

\begin{document}
\title[On Kac-Weisfeiler modules]{On Kac-Weisfeiler modules for general and special linear Lie superalgebras}
\author{Yang Zeng and Bin Shu}
\thanks{\nonumber{{\it{Mathematics Subject Classification}} (2000 {\it{revision}})
Primary 17B35. Secondary 17B50. This work is  supported partially by the NSF of China (No. 11271130; 11201293; 111126062),  the Innovation Program of Shanghai Municipal Education Commission (No. 12zz038). }}
\address{School of Science, Nanjing Audit University, Nanjing, Jiangsu Province 211815, China}
\email{zengyang214@163.com}
\address{Department of Mathematics, East China Normal University, Shanghai 200241, China}
\email{bshu@math.ecnu.edu.cn}
\begin{abstract}
\noindent Let $\ggg:=\gl_{m|n}$  be a general linear Lie superalgebra
over an algebraically closed field $\mathds{k}=\overline{\mathbb{F}}_p$ of characteristic $p>2$. A module of $\ggg$ is said to be of Kac-Weisfeiler if its dimension coincides with the dimensional lower bound in the super Kac-Weisfeiler property presented by Wang-Zhao in \cite{WZ}. In this paper, we verify the existence of the Kac-Weisfeiler modules for $\gl_{m|n}$. We also establish the corresponding consequence for the special linear Lie superalgebras $\mathfrak{sl}_{m|n}$ with restrictions $p>2$ and $p\nmid(m-n)$.
\end{abstract}
\maketitle

\setcounter{section}{-1}
\section{Introduction}
\subsection{}\label{0.1}
Let $\mathfrak{g}={\ggg}_{\bar0}+{\ggg}_{\bar1}$ be a basic Lie superalgebra over an algebraically closed field $\mathds{k}=\overline{\mathbb{F}}_p$ of characteristic $p>2$, and $U(\mathfrak{g})$ be the universal enveloping algebra of $\mathfrak{g}$. For any $p$-character $\xi\in\mathfrak{g}^*_{\bar{0}}$, the reduced enveloping algebra $U_\xi(\mathfrak{g})$ is defined as the quotient of $U(\mathfrak{g})$ by its two-sided ideal generated by all even central elements $x^p-x^{[p]}-\xi(x)^p$ with $x\in\mathfrak{g}_{\bar{0}}$. For every irreducible $\mathfrak{g}$-module $V$, there is a unique linear function $\xi:=\xi_V\in\mathfrak{g}^*_{\bar{0}}$ such that $V$ is a $U_\xi(\mathfrak{g})$-module. Write $\mathfrak{g}^\xi$ for the centralizer of $\xi$ in $\ggg$ and denote by $d_i:=\text{dim}\,\ggg_i-\dim \ggg_i^\xi$ for $i\in\mathbb{Z}_2$.

As an extension of the Kac-Weisfeiler conjecture (proved by Premet in \cite{P1}) for the modular representations of reductive Lie algebras, Wang-Zhao formulated the super Kac-Weisfeiler (KW for short) property for a basic Lie superalgebra $\mathfrak{g}$ in \cite{WZ}. It shows that for arbitrary $\xi\in\mathfrak{g}^*_{\bar{0}}$, the dimension of every $U_\xi(\mathfrak{g})$-module is divisible by $p^{\frac{d_0}{2}}2^{\lfloor\frac{d_1}{2}\rfloor}$,
where $\lfloor\frac{d_1}{2}\rfloor$ denotes the least integer upper bound of $\frac{d_1}{2}$. Here we make a proposal that in the category of $U_\xi(\ggg)$-modules, {\sl{an object is said to be of Kac-Weisfeiler if its dimension coincides with the low-bound $p^{\frac{d_0}{2}}2^{\lfloor\frac{d_1}{2}\rfloor}$}}. What we are concerned with in the present paper is the problem whether the Kac-Weisfeiler modules exist, parallel to the one for reductive Lie algebras in prime characteristic.

\subsection{}
Let us recall the same story for reductive Lie algebras in prime characteristic.
Let $\ggg_0$  be the Lie algebra of a simple and simply connected algebraic $\mathds{k}$-group $G_0$. One of the challenging open problems in the representation theory of $\ggg_0$ is to verify that for every $\xi\in\ggg_0^*$ the reduced enveloping algebra $U_\xi(\ggg_0)$ has a simple module of dimension $p^{\frac{\dim\,\co_\xi}{2}}$, where $\co_\xi$ denotes  the $G_0$-coadjoint orbit of $\xi$ in $\ggg_0^*$; see \cite[\S8]{Hu}, for example. In fact, one can show that the number $\dim\co_\xi$ is equal to $\frac{d_0}{2}$, which is exactly the dimensional lower bound for the representations of the $\mathds{k}$-algebra $U_\xi(\mathfrak{g}_0)$ in the celebrated Kac-Weisfeiler conjecture \cite{KW}. As pointed by Premet in \cite{P7}, the answer to this question is positive for the Lie algebras of type $A$. Actually, one can easily construct such a representation by inducing from a one-dimensional representation of an appropriate parabolic subalgebra in this case. In \cite[Corollary 2.3]{P7}, Premet further gave positive answers for the Lie algebras of types $B, C, D$ under the assumption that $p\gg0$. It is notable that the proof given by Premet is critically depend on some deep results of representations of finite $W$-algebras over the field of complex numbers.

\subsection{}
Return to the situation of basic Lie superalgebras. Let $\mathfrak{g}$ be the general linear Lie superalgebra $\mathfrak{gl}_{m|n}$ or the special linear Lie superalgebra $\mathfrak{sl}_{m|n}$ over $\mathds{k}$ under the assumption that  $p\nmid(m-n)$ for $\mathfrak{g}=\mathfrak{sl}_{m|n}$. In the present paper, we will give a positive answer to the question addressed in \S\ref{0.1}. The tool is some artful application of the pyramid introduced in \cite{EK} and its generalization in \cite[\S7]{H}, along with the super version of Kac-Weisfeiler's Morita equivalence (cf. \cite{WZ}).

Let us explain our approach explicitly. In virtue of the generalized pyramids with each box  marked with a ``$+$'' (resp. ``$-$'') sign, the gradings of $\mathfrak{gl}_{m|n}$ are constructed (cf. \cite[\S7]{H}). Following Brundan-Goodwin's treatment to the Lie algebra case (c.f. \cite[\S6]{BG3}), one can fix  a numbering $\overline1,\overline2,\cdots,\overline m$ (resp. $1,2,\cdots,n$) of the boxes with a ``$+$'' (resp. ``$-$'') sign for the generalized pyramids, and define the corresponding elements and subalgebras in $\mathfrak{gl}_{m|n}$.

With the above setup, we construct the ``shifted pyramids" associated with a given pair of partitions,  from their symmetric ones of $\mathfrak{gl}_{m|n}$ (a symmetric pyramid will be called a Dynkin one in the paper). Associate with the ``shifted pyramid", we construct a parabolic subalgebra of $\mathfrak{gl}_{m|n}$ such that it has a one-dimensional module. Then by inducing this module we obtain a representation of $\mathfrak{gl}_{m|n}$ which turns to be a Kac-Weisfeiler module. The Kac-Weisfeiler modules of $\mathfrak{sl}_{m|n}$ can also be constructed in the same way, since each element of $\mathfrak{gl}_{m|n}$ can be considered as an element in $\mathfrak{sl}_{m|n}$ up to a scalar multiple of unitary matrix.

\subsection{}

The paper is organized as follows. In \S\ref{1}, we recall some basics on the Dynkin gradings for a basic Lie superalgebra $\mathfrak{g}$, along with the super  Kac-Weisfeiler property.  The contents of \S\ref{2} are devoted to the connections between the (shifted) pyramids and the representations of $\mathfrak{gl}_{m|n}$. We then construct in \S\ref{3} a class of Kac-Weisfeiler modules with dimensions equal to the lower bounds introduced in \S\ref{0.1} based on the consequences obtained in \S\ref{2}.

In the forthcoming paper \cite{ZS2}, we will certify the existence of the Kac-Weisfeiler modules for all basic Lie superalgebras under some assumption. Our arguments will be critically dependent on the theory of finite $W$-superalgebras developed in \cite{ZS}, along with an assumption which is conjectured to be fully satisfied, saying that the smallest dimension for the representations of a finite $W$-superalgebra over the field of complex numbers is one or two, according to its detecting parity  (Although this conjecture has not been settled in full generality, we can still establish it in a lot of special cases, including all the Lie superalgebras of type $A$ based on the results of the present paper).

\subsection{}
Throughout we work over $\dsk$ an algebraically closed field of characteristic $p>2$ as the ground field.
Let ${\bbz}_+$ be the set of all the non-negative integers in ${\bbz}$, and denote by $\mathbb{Z}_2$ the residue class ring modulo $2$ in $\mathbb{Z}$. A superspace is a $\mathbb{Z}_2$-graded vector space $V=V_{\bar0}\oplus V_{\bar1}$, in which we call elements in $V_{\bar0}$ and $V_{\bar1}$ even and odd, respectively.  Write $|v|\in{\bbz}_2$ for the parity (or degree) of $v\in V$, which is implicitly assumed to be ${\bbz}_2$-homogeneous. We will use the notations
$$\text{\underline{dim}}V=(\text{dim}V_{\bar0},\text{dim}V_{\bar1}),\quad\text{dim}V=\text{dim}V_{\bar0}+\text{dim}V_{\bar1}.$$

By vector spaces, subalgebras, ideals, modules, and submodules etc.,  we mean in the super sense unless otherwise specified, throughout the paper.

\section{preliminary}\label{1}
\subsection{Dynkin gradings for $\mathfrak{gl}_{m|n}$ and $\mathfrak{sl}_{m|n}$}\label{1.1}

Let $\mathfrak{g}=\mathfrak{gl}_{m|n}$ with $m,n\in\mathbb{Z}_+$. For a given nilpotent element $e$ in $\ggg$, we associate with a partition $(q_m,q_n)$ of $(m|n)$, where $q_i = (q_i(1),\cdots,q_i(r_i))$ of length $r_i$ is the shape of the Jordan canonical form of the summand of $e$ in $\gl_i$ for $i=m,n$ respectively. Denote $(\mbox{GL}_{m|n})_\ev:=\mbox{GL}(m,\bbk)\times \mbox{GL}(n,\bbk)$, then there is a one-to-one correspondence between $(\mbox{GL}_{m|n})_\ev$-orbits of nilpotent even elements in $\gl_{m|n}$ and partitions of $(m|n)$.

Let $V = V_\bz\oplus V_\bo$ be a vector space with $\underline{\dim}\,V = (m|n)$. Identify $\gl_{m|n}$ with $\gl(V)$, whose even
part is $\gl(V)_\bz = \gl(V_\bz) \oplus \gl(V_\bo)$. Then the nilpotent element $e$ mentioned above is actually a nilpotent transformation on $V$. Associated with $e$, there exist $v_1,\cdots,v_{r_m}\in V_\bz$ and  $u_1,\cdots, u_{r_n}\in V_\bo$ such that all $e^jv_s$ and $e^lu_t$ with $1\leqslant  s \leqslant  r_m$, $1\leqslant t\leqslant  r_n$, and $0\leqslant  j< q_m(s)$, $0\leqslant  l<q_n(t)$, are the basis for $V_\bz$ and $V_\bo$ such that $e^{q_m(s)}v_s = 0$, $e^{q_n(t)}u_t=0$. One can endow $V$  with a $\bbz$-graded structure  via defining $V(k)_\bz$ and $V(k)_\bo$ equal to the span of $e^jv_s$ with $k = 2j+1-q_m(s)$ and $e^ju_t$ with $k = 2j+1-q_n(t)$, respectively.

Now we will introduce a special $\mathbb{Z}$-grading for $\mathfrak{g}$. Consider a cocharacter $\tau = \tau_0\times\tau_1$ which is a homomorphism of algebraic groups from $\dsk^\times(=\dsk\backslash\{0\})$ to $\mbox{GL}(V)_\ev:= \mbox{GL}(V_\bz) \times \mbox{GL}(V_\bo)$ defined as follows. For each $t\in \dsk^\times$, let $\tau(t) = \tau_0(t) \times\tau_1(t)$ be the linear map such that $\tau_i(t)(v) = t^kv$ for all $v\in V(k)_i:=V(k)\cap V_i$ and all $k\in\bbz$ with $i\in\{0,1\}$. (Here we identify $\{0,1\}$ with $\{\bz,\bo\}$ for the vector spaces $V_{\bar0}$ and $V_{\bar1}$, respectively).
Now by the arguments in \cite[Remark 3.4]{WZ}, one can define a $\bbz$-grading on $\ggg$ via setting \begin{align}\label{Dynkingrading}
\ggg=\bigoplus_{k\in\bbz}\ggg(k)\mbox{ with }
&\ggg(k) = \{Z \in\gl(V)\mid \mbox{Ad}(\tau(t))(Z) = t^kZ\mbox{ for all }t\in \dsk^\times\}.
\end{align}
The above grading is called the Dynkin grading associated to the nilpotent element $e$ of partition $(q_m,q_n)$.

Recall \cite[Proposition 2.1]{ZS} shows that there exists an even non-degenerate supersymmetric invariant bilinear form $(\cdot,\cdot)$. If we write $\ggg^e$ for the centralizer of $e$ in $\ggg$, then the Dynkin grading satisfies the following properties:

\begin{prop}(\cite[Theorem 3.1]{WZ})\label{Dynkingradingprop}
The following are true:
\begin{itemize}
\item[(1)] $\text{ad}\,e:~\mathfrak{g}(j)\rightarrow\mathfrak{g}(j+2)~\text{is injective for}~j\leqslant-1;$
\item[(2)] $\text{ad}\,e:~\mathfrak{g}(j)\rightarrow\mathfrak{g}(j+2)~\text{is surjective for}~j\geqslant-1;$
\item[(3)] the Dynkin grading is compatible with the $\bbz_2$-grading; this is to say, $\ggg(k)=\ggg(k)_\bz+\ggg(k)_\bo$ with $\ggg(k)_i=\ggg(k)\cap\ggg_i$ for $i\in\bbz_2$, $k\in\bbz$;
\item[(4)] $e\in \ggg(2)_\bz$;  $\ggg^e(s)=0$ if $s<0$;
\item[(5)] $\underline{\dim}\,\ggg^e=\underline{\dim}\,\ggg(0)+\underline{\dim}\,\ggg(1)$;
\item[(6)] $(\ggg(k),\ggg(l))=0$ if $k+l\ne 0$; $\underline{\dim}\,\ggg(k)=\underline{\dim}\,\ggg(-k)$ for all $k\in \bbz$.
\end{itemize}
\end{prop}

Recall $d_i=\dim\ggg_i-\dim\ggg^e_i$ for $i\in\bbz_2$ by the notation in \S\ref{0.1}, and it follows from Proposition \ref{Dynkingradingprop} that
\begin{equation}\label{dim}
\text{\underline{dim}}\,\ggg-\text{\underline{dim}}\,\ggg^e=\sum\limits_{k\geqslant2}
2\text{\underline{dim}}\,\ggg(-k)+\text{\underline{dim}}\,\ggg(-1).
\end{equation}
Therefore, $\text{dim}\,\ggg(-1)_{\bar1}$ and $d_1$ always share the same parity.
Set $\mathfrak{p}:=\bigoplus_{k\geqslant 0}\ggg(k)$ to be a parabolic subalgebra of $\ggg$. Define $\chi\in\ggg^{*}$ by letting $\chi(x)=(e,x)$ for $x\in\ggg$, then we have $\chi(\ggg_{\bar{1}})=0$. Since the bilinear form $(\cdot,\cdot)$ is invariant, it is easy to verify that $\chi(\ggg(k))=(e,\ggg(k))=0$ for all $k\neq-2$. In particular, $\chi(\ppp)=0$.

\begin{rem}
Since each element of $\mathfrak{gl}_{m|n}$ can be considered as an element in $\mathfrak{sl}_{m|n}$ up to a scalar multiple of unitary matrix, all the discussions in \S\ref{1.1} go through for the special linear Lie superalgebras $\mathfrak{sl}_{m|n}$, with a few modifications.
\end{rem}

\subsection{Super Kac-Weisfeiler Property}
Let $\ggg$ be a basic Lie superalgebra (c.f. \cite[\S2.2]{WZ}), and denote by $U_\xi(\ggg)$ the reduced enveloping algebra of $\ggg$ with $p$-character $\xi\in\ggg^*_{\bar0}$. In \cite[Theorem 4.3]{WZ} Wang-Zhao introduced the super Kac-Weisfeiler property with nilpotent $p$-characters for a basic Lie superalgebra $\ggg$ over positive characteristic field $\mathds{k}=\overline{\mathbb{F}}_p$ (with some restrictions on $p$, c.f. \cite[Table 1]{WZ}), i.e.
\begin{prop}(\cite{WZ})\label{wzd}
Let $\ggg$ be a basic Lie superalgebra and $\chi\in\ggg^*_{\bar0}$ be nilpotent. Then the dimension of every $U_\chi(\ggg)$-module $M$ is divisible by $p^{\frac{d_0}{2}}2^{\lfloor\frac{d_1}{2}\rfloor}$.
\end{prop}

In virtue of this result, the super Kac-Weisfeiler property with any $p$-characters of $\ggg$ was also formulated by them in \cite[Theorem 5.6]{WZ}, as cited as below:

\begin{prop}(\cite{WZ})\label{wzd2}
Let $\ggg$ be a basic Lie superalgebra, and let $\xi$ be arbitrary $p$-character in $\ggg^*_{\bar0}$. Then the dimension of every $U_\xi(\ggg)$-module $M$ is divisible by $p^{\frac{d_0}{2}}2^{\lfloor\frac{d_1}{2}\rfloor}$.
\end{prop}

\begin{rem}\label{even}
(1) Although the general linear Lie superalgebras are not of basic Lie superalgebras, Proposition~\ref{wzd} and Proposition~\ref{wzd2} are still valid for $\mathfrak{gl}_{m|n}$.

(2) In \cite[\S3.2]{WZ}, Wang-Zhao gave an explicit description for the Dynkin gradings of $\ggg=\mathfrak{gl}_{m|n}$. In particular, they computed the dimension of $(\mathfrak{gl}_{m|n})^e_{\bar1}$ for any even nilpotent element $e\in(\mathfrak{gl}_{m|n})_{\bar0}$ and showed that it is always an even number. As the dimension of $(\mathfrak{gl}_{m|n})_{\bar1}$ is always even, by the definition one knows that $d_1$ is also an even number, i.e. $\lfloor d_1\rfloor=d_1$. The same discussion also works for $\ggg=\mathfrak{sl}_{m|n}$.
\end{rem}

\section{Pyramids for $\mathfrak{gl}_{m|n}$}\label{2}
In this section we will first recall some material on the pyramids for $\ggg:=\mathfrak{gl}_{m|n}$. One can refer to the references \cite{BG3}, \cite{EK}, \cite{H}  and \cite{Peng3} for more details on pyramids. We will then introduce the so-called ``shifted pyramids" associated to the Dynkin pyramids of $\mathfrak{gl}_{m|n}$ which will be crucially useful to the sequent arguments.

\subsection{Pyramids for $\mathbb{Z}$-gradings of $\mathfrak{gl}_{m|n}$}\label{ordinarypyramid}
A pyramid $P$ is a diagram, which is by the definition a shaped collection of finite boxes of size $2\times2$ in the upper half plane satisfying: (1) all boxes are centered at integer coordinates; (2) for each $j=1,\cdots,N$ (the number of boxes of $P$), the second coordinates of the $j${th} row equal $2j-1$ and the first coordinates of
the $j${th} row form an arithmetic progression $f_j ,f_{j+2},\cdots,l_j$ with difference $2$; (3) the first row is centered at $(0,0)$, i.e. $f_1=-l_1$, and $$f_j\leqslant f_{j+1}\leqslant l_{j+1}\leqslant l_j$$ for all $j$. We will further endow each row of boxes of $P$ with parity associated with nilpotent element $e\in (\gl_{m|n})_\bz$.

At first, we say that $P$ has size $(m|n)$ if $P$ has exactly $m$ even boxes and $n$ odd boxes. Fix $m,n\in\mathbb{Z}_+$ and let $(r,q)$ be a partition of $(m|n)$.  Let $z=\psi(r,q)\in Par(m+n)$ be the total ordered combination of partitions $r\vdash m$ and $q\vdash n$ with ordering ``$\geqslant$" (and also ``$>$") which satisfies: (1) both orders on $r$ and $q$ are preserved in $z$; (2) if $r_i=q_j$ for some $i,j$ then $\psi(q_j)>\psi(r_i)$. We define $Pyr(r,q)$ to be the set of pyramids which satisfy the following two conditions: (3) the $j${th} row of a pyramid $P\in Pyr(r,q)$ has length $z_j$, in connection to the partition $z=(z_1\geqslant z_2\geqslant\cdots \geqslant z_{s+t})\vdash m+n$; (4) if $\psi^{-1}(z_j)\in r$ (resp. $\psi^{-1}(z_j)\in q$), then all boxes in the $j${th} row have even (resp. odd) parity and we mark these boxes with a ``$+$'' (resp. ``$-$'') sign.

Let $I=\{\overline1<\cdots<\overline m<1<\cdots<n\}$ be an ordered index set. For the pyramid $P\in Pyr(r,q)$, we fix once and for all a numbering $\overline1,\overline2,\cdots,\overline{m}$ (resp. $1,2,\cdots,n$) of the boxes with a ``$+$'' (resp. ``$-$'') sign, from left to right along a row, then the next row up.
Let $\row(i)$ and $\col(i)$ denote the row and column numbers of the $i$th box.

Writing $e_{i,j}$ for the $(i,j)$-matrix unit, we define a nilpotent element $e(P)=\sum\limits_{i,j}e_{i,j}\in(\gl_{m|n})_{\bar0}$ with $i,j$ summing over all the pairs in $\{\overline1,\cdots,\overline m,1,\cdots,n\}$ such that $\row(i) = \row(j)$ and $\col(j)-\col(i)=2$. Then $e(P)$ is a nilpotent matrix, exactly of standard Jordan form associated with the pair of partitions $(r,q)$. Since $e(P)$ does not depend on the choice of $P$ in $Pyr(r,q)$ by the definition, {\sl{we may denote it by $e_{r,q}$}}. Moreover, $e_{r,q}\in\ggg_{\bar0}$ since all the boxes in the same row have the same parity in the super structure of $\ggg$.

Furthermore, for a given pyramid $P\in Pyr(r,q)$ there is an associated $\mathbb{Z}$-graded structure on $\ggg$: $\ggg=\bigoplus_{k\in\mathbb{Z}}\ggg^P(k)$ such that $\ggg^P(k)$ is a $\bbk$-span of those matrix units of degree $k$ in the sense that the $(i,j)$-matrix unit $e_{i,j}$ is of degree $\col(j)-\col(i)$. Obviously we have $-|m+n|\leqslant\text{deg}\,e_{i,j}\leqslant|m+n|$.  Moreover, associated to the pyramid $P$ one can define a parabolic subalgebra $\mathfrak{p}=\sum\limits_{\begin{subarray}{c}i,j\in\{\overline1,\cdots,\overline m,1,\cdots,n\},\\ \col(i)\leqslant \col(j)\end{subarray}}\bbk e_{i,j}$. It can be observed that $\mathfrak{p}$ is a restricted subalgebra of $\ggg$.

For the case when the pyramid $P\in Pyr(r,q)$ is symmetric (along the vertical line across $(0,0)$), the associated $\mathbb{Z}$-grading coincides with the Dynkin grading (\ref{Dynkingrading}) by the definition. We restate it as below.

\begin{lemma} \label{two gradings} Let $e_{r,q}$ be the nilpotent matrix of standard Jordan form of type $(r,q)$ defined as above. Then the corresponding Dynkin grading of $\ggg$ arising from $e_{r,q}$ coincides with the one arising from the symmetric pyramid of type $(r,q)$.
  \end{lemma}
  \begin{proof}
 Actually, 
  one has a cocharacter $\tau_{r,q}:=\tau_r\times\tau_q$ as presented in \S\ref{1.1}, which directly defines the Dynkin grading of $\ggg$ for the given nilpotent element $e_{r,q}$. Then the result can be obtained readily by the matrix computation.
  \end{proof}

Henceforth, for a given partition $(r,q)$, {\sl{ the symmetric pyramid associated to which will be called a Dynkin one, and denote it by $P_{r,q}$.}}

\subsection{Shifted pyramids for $\mathfrak{gl}_{m|n}$}\label{shiftedpyramid}
Now we construct a new pyramid from the Dynkin pyramid as follows: for a given Dynkin pyramid $P_{r,q}$, shift some rows leftward by one unit as long as those rows admit the lengths whose parity is different from the one of the length of the first row, then one can obtain a new pyramid by the definition, and denote it by $P'$. We will call such pyramid $P'$ resulted from the shifting carried above ``the shifted pyramid" of the Dynkin pyramid $P_{r,q}$. Obviously, we have $P'\in Pyr(r,q)$.

For each box of $P_{r,q}$ (resp. $P'$) filled with number $i\in\{\overline1,\cdots,\overline m,1,\cdots,n\}$, denote by $\row(i)$ (resp. $\row'(i)$) and $\col(i)$ (resp. $\col'(i)$) the row and column numbers of the $i${th} box.
Without loss of generality, we might as well assume that $r_1\geqslant q_1$ for simplicity. 
Then we can easily conclude that:
\begin{itemize}
\item[(1)] $\col'(i)=\col(i)$
   in the case when the length of $\row(i)${th} row is of the same parity as that of $r_1$.  By a straightforward calculation, it can be easily seen that $\col(i)=\col'(i)$ must admit a parity different from that of $r_1$ in this case.
\item[(2)] $\col'(i)=\col(i)-1$ in the case when 
the length of $\row(i)${th} row admits the parity different from that of $r_1$. Then $\col(i)$ must have the same parity as that of $r_1$. Hence $\col'(i)$ must have the different parity as that of $r_1$.
\end{itemize}
Summing up, we conclude that for any $i\in\{\overline1,\cdots,\overline m,1,\cdots,n\}$, $\col'(i)$ always admits the different parity from that of $r_1$ in the shifted pyramid $P'$.  On the other hand, for both degrees in the same sense as before, we have
 \begin{align*}\text{deg}\,e_{i,j}&=\col(j)-\col(i), &
  \text{deg}'e_{i,j}&=\col'(j)-\col'(i),
  \end{align*}
 associated with the pyramids $P_{r,q}$ and $P'$ respectively. Then we can define both gradations of $\ggg$ associated with the Dynkin pyramid and its shifted one respectively. By the remark just summarised, $\text{deg}'e_{i,j}$, the degree of any $(i,j)$-matrix unit $e_{i,j}$ associated with the shifted pyramid $P'$, is always even. Generally,  we call {\sl{a pyramid to be even if the degrees of the $(i,j)$-matrix unit $e_{i,j}$ associated with this pyramid are all even for $i, j\in\{\overline1,\cdots,\overline m,1,\cdots,n\}$}}. Hence, we have
\begin{lemma} \label{evenlem} Any shifted pyramid is always an even pyramid.
\end{lemma}

We continue to set
\begin{align}\label{pp'2}
 &\mathfrak{p}=\sum\limits_{\substack{
 i,j\in\{\overline1,\cdots,\overline m,1,\cdots,n\},\\\col(i)\leqslant \col(j)}}\bbk e_{i,j},&\mathfrak{p}'=\sum\limits_{\substack{
 i,j\in\{\overline1,\cdots,\overline m,1,\cdots,n\},\\\col'(i)\leqslant \col'(j)}}\bbk e_{i,j}.
 \end{align}
 Recall Lemma \ref{two gradings} shows that the grading arising from the Dynkin pyramid coincides with the Dynkin grading, thus $\mathfrak{p}=\bigoplus_{i\geqslant0}\ggg^{P_{r,q}}(i)$.
     On the other side, for the shifted pyramid $P'$ we have $\mathfrak{p}'=\bigoplus_{i\geqslant0}\ggg^{P'}(i)$ by the definition.

\begin{example}\label{expyr1}
Let $\ggg=\mathfrak{gl}_{4|3}$ be a general linear Lie superalgebra and consider the partitions $r=(3,1)$ and $q=(2,1)$, with the corresponding nilpotent element
$$e(r,q)=e_{\overline1,\overline2}+e_{\overline2,\overline3}+e_{1,2}\in\ggg(2)_{\bar0}.$$
  The Dynkin grading of $\ggg_{\bar{0}}\cong\mathfrak{gl}_4\oplus\mathfrak{gl}_3$ for the partition $(r,q)$ corresponds to the following symmetric pyramids:
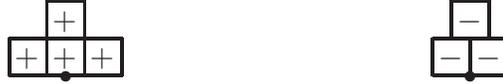
\setlength{\unitlength}{5mm}
\begin{center}
\begin{picture}(3,2)(0,0)
\linethickness{1pt}
\put(0,0){\line(1,0){3}}
\put(0,1){\line(1,0){3}}
\put(1,2){\line(1,0){1}}
\put(0,0){\line(0,1){1}}
\put(1,0){\line(0,1){2}}
\put(2,0){\line(0,1){2}}
\put(3,0){\line(0,1){1}}
\put(1.5,0){\circle*{0.3}}
\put(0.15,0.25){$+$}
\put(1.15,0.25){$+$}
\put(2.15,0.25){$+$}
\put(1.15,1.25){$+$}
\end{picture}\qquad\qquad\qquad\qquad\qquad
\begin{picture}(3,2)(0,0)
\linethickness{1pt}
\put(0,0){\line(1,0){2}}
\put(0,1){\line(1,0){2}}
\put(0.5,2){\line(1,0){1}}
\put(0,0){\line(0,1){1}}
\put(0.5,1){\line(0,1){1}}
\put(1,0){\line(0,1){1}}
\put(1.5,1){\line(0,1){1}}
\put(2,0){\line(0,1){1}}
\put(1,0){\circle*{0.3}}
\put(0.15,0.25){$-$}
\put(1.15,0.25){$-$}
\put(0.65,1.25){$-$}
\end{picture}
\end{center}

Fill the boxes of a "$+$'' (resp. "$-$'') sign with numbers $\overline1, \overline2, \overline3, \overline4$ (resp. $1, 2, 3$) from left to right along a row, and repeat  upwards, then
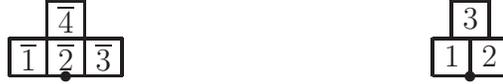
\begin{center}
\begin{picture}(3,2)(0,0)
\linethickness{1pt}
\put(0,0){\line(1,0){3}}
\put(0,1){\line(1,0){3}}
\put(1,2){\line(1,0){1}}
\put(0,0){\line(0,1){1}}
\put(1,0){\line(0,1){2}}
\put(2,0){\line(0,1){2}}
\put(3,0){\line(0,1){1}}
\put(1.5,0){\circle*{0.3}}
\put(0.3,0.15){$\overline1$}
\put(1.3,0.15){$\overline2$}
\put(2.3,0.15){$\overline3$}
\put(1.3,1.15){$\overline 4$}
\end{picture}\qquad\qquad\qquad\qquad\qquad
\begin{picture}(3,2)(0,0)
\linethickness{1pt}
\put(0,0){\line(1,0){2}}
\put(0,1){\line(1,0){2}}
\put(0.5,2){\line(1,0){1}}
\put(0,0){\line(0,1){1}}
\put(0.5,1){\line(0,1){1}}
\put(1,0){\line(0,1){1}}
\put(1.5,1){\line(0,1){1}}
\put(2,0){\line(0,1){1}}
\put(1,0){\circle*{0.3}}
\put(0.3,0.25){$1$}
\put(1.3,0.25){$2$}
\put(0.8,1.25){$3$}
\end{picture}
\end{center}

Furthermore, we can construct pyramids for $Pyr(r,q)$, which are induced from the above pyramids, as follows
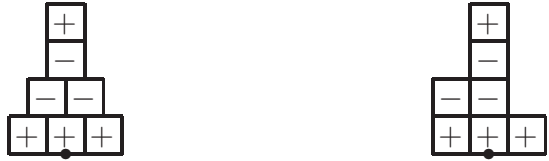
\begin{center}
\begin{picture}(3,4)(0,0)
\linethickness{1pt}
\put(0,0){\line(1,0){3}}
\put(0,1){\line(1,0){3}}
\put(0.5,2){\line(1,0){2}}
\put(1,3){\line(1,0){1}}
\put(1,4){\line(1,0){1}}
\put(0,0){\line(0,1){1}}
\put(0.5,1){\line(0,1){1}}
\put(1,0){\line(0,1){1}}
\put(1,2){\line(0,1){2}}
\put(1.5,1){\line(0,1){1}}
\put(2,0){\line(0,1){1}}
\put(2,2){\line(0,1){2}}
\put(2.5,1){\line(0,1){1}}
\put(3,0){\line(0,1){1}}
\put(1.5,0){\circle*{0.3}}
\put(0.15,0.25){$+$}
\put(1.15,0.25){$+$}
\put(2.15,0.25){$+$}
\put(0.65,1.25){$-$}
\put(1.65,1.25){$-$}
\put(1.15,2.25){$-$}
\put(1.15,3.25){$+$}
\end{picture}\qquad\qquad\qquad\qquad\qquad
\begin{picture}(3,4)(0,0)
\linethickness{1pt}
\put(0,0){\line(1,0){3}}
\put(0,1){\line(1,0){3}}
\put(0,2){\line(1,0){2}}
\put(1,3){\line(1,0){1}}
\put(1,4){\line(1,0){1}}
\put(0,0){\line(0,1){2}}
\put(1,0){\line(0,1){4}}
\put(2,0){\line(0,1){4}}
\put(3,0){\line(0,1){1}}
\put(1.5,0){\circle*{0.3}}
\put(0.15,0.25){$+$}
\put(1.15,0.25){$+$}
\put(2.15,0.25){$+$}
\put(0.15,1.25){$-$}
\put(1.15,1.25){$-$}
\put(1.15,2.25){$-$}
\put(1.15,3.25){$+$}
\end{picture}
\end{center}

One can easily observe that the first pyramid above corresponds to the Dynkin grading, and the second one is exactly the shifted pyramid induced from the first one. Now fill the boxes of a "$+$'' (resp. "$-$'') sign with numbers $\overline1,\overline2,\overline3,\overline4$ (resp. $1,2,3$), then
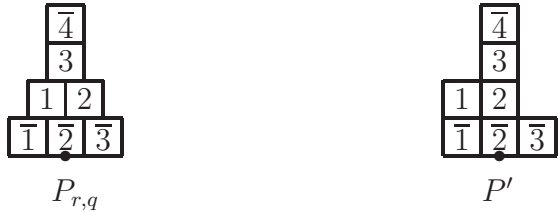
\begin{center}
\begin{picture}(3,4)(0,0)
\linethickness{1pt}
\put(0,0){\line(1,0){3}}
\put(0,1){\line(1,0){3}}
\put(0.5,2){\line(1,0){2}}
\put(1,3){\line(1,0){1}}
\put(1,4){\line(1,0){1}}
\put(0,0){\line(0,1){1}}
\put(0.5,1){\line(0,1){1}}
\put(1,0){\line(0,1){1}}
\put(1,2){\line(0,1){2}}
\put(1.5,1){\line(0,1){1}}
\put(2,0){\line(0,1){1}}
\put(2,2){\line(0,1){2}}
\put(2.5,1){\line(0,1){1}}
\put(3,0){\line(0,1){1}}
\put(1.5,0){\circle*{0.3}}
\put(0.3,0.15){$\overline1$}
\put(1.3,0.15){$\overline2$}
\put(2.3,0.15){$\overline3$}
\put(0.8,1.25){$1$}
\put(1.8,1.25){$2$}
\put(1.3,2.25){$3$}
\put(1.3,3.15){$\overline4$}
\end{picture}
\qquad\qquad\qquad\qquad\qquad
\begin{picture}(3,4)(0,0)
\linethickness{1pt}
\put(0,0){\line(1,0){3}}
\put(0,1){\line(1,0){3}}
\put(0,2){\line(1,0){2}}
\put(1,3){\line(1,0){1}}
\put(1,4){\line(1,0){1}}
\put(0,0){\line(0,1){2}}
\put(1,0){\line(0,1){4}}
\put(2,0){\line(0,1){4}}
\put(3,0){\line(0,1){1}}
\put(1.5,0){\circle*{0.3}}
\put(0.3,0.15){$\overline1$}
\put(1.3,0.15){$\overline2$}
\put(2.3,0.15){$\overline3$}
\put(0.3,1.25){$1$}
\put(1.3,1.25){$2$}
\put(1.3,2.25){$3$}
\put(1.3,3.15){$\overline4$}
\end{picture}
\end{center}
\begin{center}
$P_{r,q}$ \quad \qquad\qquad\qquad\qquad\qquad\quad $P'$
\end{center}

The parabolic subalgebras  $\ppp$ and $\ppp'$ associated with $P_{r,q}$ and $P'$ respectively can be easily listed. Consequently, $\ppp'=\ppp\oplus(\bbk e_{1\overline1}+\bbk e_{2\overline2}+\bbk e_{23}+\bbk e_{2\overline4})$.
\end{example}

\subsection{Parabolic subalgebras of shifted pyramids} We continue the arguments on the shifted pyramids. For a given nilpotent element $e\in\ggg_\bz=(\gl_{m|n})_\bz$, we know that it is uniquely represented by a  partition  $(r,q)=(r_1,r_2,\cdots,r_s;q_1,q_2,\cdots,q_t)$ of $(m|n)$ with $r_1\geqslant r_2\geqslant \cdots\geqslant r_s$ and $q_1\geqslant q_2\geqslant \cdots\geqslant q_t$, under the conjugation of $(\mbox{GL}_{m|n})_\ev=\mbox{GL}(m,\bbk)\times \mbox{GL}(n,\bbk)$. From now on, we always assume $e:=e_{r,q}$ is just a nilpotent matrix of Jordan canonical form of type $(r,q)$. In the sequent arguments, we still assume that $r_1\geqslant q_1$ for simplicity.

As discussed in \S\ref{shiftedpyramid}, associated with $e$ we can construct the Dynkin pyramid $P_{r,q}$ and its shifted one $P'$. Then we have

\begin{lemma}\label{eo} Let $i, j\in\{\overline1,\cdots,\overline m,1,\cdots,n\}$. For the partition $(r,q)=(r_1,r_2,\cdots,r_s;$\\$q_1,q_2,\cdots,q_t)$ of $(m|n)$ with $r_1\geqslant q_1$, the following statements hold:
\begin{itemize}
\item[(1)] if $\text{deg}\,e_{i,j}$ is even,
    then $\text{deg}'e_{i,j}=\text{deg}\,e_{i,j}$;
\item[(2)] if $\text{deg}\,e_{i,j}$ is odd,
then $\text{deg}'e_{i,j}=\text{deg}\,e_{i,j}\pm1$.
    \end{itemize}
\end{lemma}

\begin{proof}
First recall that $\text{deg}\,e_{i,j}=\col(j)-\col(i)$ and $\text{deg}'e_{i,j}=\col'(j)-\col'(i)$. We proceed our arguments in different cases. 

(1) When $\text{deg}\,e_{i,j}$ is even, then $\col(i)$ and $\col(j)$ have the same parity. By the discussion in \textsection\ref{shiftedpyramid}, one can conclude that:

(i) if both $\col(i)$ and $\col(j)$ have the different parity from $r_1$, then $\col'(i)=\col(i)$ and $\col'(j)=\col(j)$. Thus $\text{deg}'e_{i,j}=\col(j)-\col(i)$, i.e. $\text{deg}'e_{i,j}=\text{deg}\,e_{i,j}$;

(ii) if both $\col(i)$ and $\col(j)$ have the same parity as $r_1$, then $\col'(i)=\col(i)-1$ and $\col'(j)=\col(j)-1$. Thus $\text{deg}'e_{i,j}=(\col(j)-1)-(\col(i)-1)=\col(j)-\col(i)$, i.e. $\text{deg}'e_{i,j}=\text{deg}\,e_{i,j}$.

(2) When $\text{deg}\,e_{i,j}$ is odd, then $\col(i)$ and $\col(j)$ have the different parity. We have the following consequence:

(i) if $\col(i)$ is of the different parity from $r_1$, and $\col(j)$ is of the same parity as $r_1$, then $\col'(i)=\col(i)$ and $\col'(j)=\col(j)-1$. Thus  $\text{deg}'e_{i,j}=(\col(j)-1)-\col(i)=\col(j)-\col(i)-1$, i.e. $\text{deg}'e_{i,j}=\text{deg}\,e_{i,j}-1$;

(ii) if $\col(i)$ is of the same parity as $r_1$,  and $\col(j)$ is of the different parity from $r_1$,  then $\col'(i)=\col(i)-1$ and $\col'(j)=\col(j)$. Thus $\text{deg}'e_{i,j}=\col(j)-(\col(i)-1)=\col(j)-\col(i)+1$, i.e. $\text{deg}'e_{i,j}=\text{deg}\,e_{i,j}+1$.
\end{proof}

Now we are in a position to introduce the main result in this section.

\begin{prop}\label{pp'}
For the partition $(r,q)=(r_1,r_2,\cdots,r_s;q_1,q_2,\cdots,q_t)$ of $(m|n)$, let $P_{r,q}$ and $P'$ be the Dynkin pyramid and its shifted one respectively.
Denote by $\mathfrak{p}$ and $\ppp'$ 
the parabolic subalgebras of $\mathfrak{gl}_{m|n}$ corresponding to the pyramids $P_{r,q}$ and $P'$ respectively.
Then we have \begin{equation}\label{keyequation}
\underline{\dim}\,\mathfrak{p}'=\underline{\dim}\,\mathfrak{p}+{1\over 2}\underline{\dim}\,\ggg^{P_{r,q}}(-1).
\end{equation}
\end{prop}

\begin{proof} Let $(r,q)=(r_1,r_2,\cdots,r_s;q_1,q_2,\cdots,q_t)$ be the partition of $(m|n)$. For simplicity of arguments, we assume that $r_1\geqslant q_1$ without loss of generality. Recall that $\text{deg}\,e_{i,j}=\col(j)-\col(i)$ and $\text{deg}'e_{i,j}=\col'(j)-\col'(i)$ for $i,j\in\{\overline1,\cdots,\overline m,1,\cdots,n\}$. Now we proceed the proof by steps.

Step 1: We claim that $\ppp\subseteq \ppp'$. Firstly recall that $\mathfrak{p}=\sum\limits_{\begin{subarray}{c}i,j\in\{\overline1,\cdots,\overline m,1,\cdots,n\}\\ \col(i)\leqslant \col(j)\end{subarray}}\bbk e_{i,j}$, and $\mathfrak{p}'=\sum\limits_{\overset{i,j\in\{\overline1,\cdots,\overline m,1,\cdots,n\}}{\col'(i)\leqslant \col'(j)}}\bbk e_{i,j}$. By Lemma \ref{eo},
\begin{align*}
\text{deg}'e_{i,j}=\begin{cases} \text{deg}\,e_{i,j}, & \mbox{ if }\deg\,{e_{i,j}} \mbox{ is even};\cr
  \text{deg}\,e_{i,j}\pm1, &\mbox{ if }\deg\,{e_{i,j}} \mbox{ is odd}.
  \end{cases}
  \end{align*}
   Hence, we have $e_{i,j}\in\mathfrak{p}'$ for all the pairs $i,j\in\{\overline1,\cdots,\overline m,1,\cdots,n\}$ satisfying $\col(j)-\col(i)\geqslant1$. Then it follows that $\bigoplus_{i\geqslant 1}\ggg^{P_{r,q}}(i)\subseteq \ppp'$.  
   Furthermore,  $\text{deg}'e_{i,j}=\text{deg}\,e_{i,j}$ when the $(i,j)$-matrix unit $e_{i,j}$ falls in $\ggg^{P_{r,q}}(0)$.    Combining this with the definition $$\ggg^{P_{r,q}}(0)=\sum\limits_{\begin{subarray}{c}i,j\in\{\overline1,\cdots,\overline m,1,\cdots,n\},\\ \col(i)=\col(j)\end{subarray}}\bbk e_{i,j},$$
      we have $\ggg^{P_{r,q}}(0)\subseteq \ppp'$. Hence we complete the proof of the claim $\ppp\subseteq \ppp'$.

 Step 2:  Next, recall that $\ggg^{P_{r,q}}(-1)=\sum\limits_{\begin{subarray}{c}i,j\in\{\overline1,\cdots,\overline m,1,\cdots,n\},\\ \col(j)-\col(i)=-1\end{subarray}}\bbk e_{i,j}$. By Lemma \ref{eo} we have the decomposition $\ggg^{P_{r,q}}(-1)=\mathfrak{l}_1\oplus\mathfrak{l}_2$, where $\mathfrak{l}_1$ denotes the $\mathds{k}$-span of those $e_{i,j}$ with $\col(j)-\col(i)=-1$ such that $\col'(j)-\col'(i)=0$ (i.e. $e_{i,j}\in\ggg^{P'}(0)$), and $\mathfrak{l}_2$ the $\mathds{k}$-span of those $e_{i,j}$ with $\col(j)-\col(i)=-1$ such that $\col'(j)-\col'(i)=-2$ (i.e. $e_{i,j}\in\ggg^{P'}(-2)$).

We claim that $\ppp'=\ppp\oplus \frak{l}_1$. In order to prove this, we first need to understand the spaces $\mathfrak{l}_1$ and $\mathfrak{l}_2$ explicitly.

Keep in mind that the shifted pyramid $P'$ is obtained from the Dynkin pyramid $P_{r,q}$ by the left translation by one unit of those rows whose lengths have the different parity from that of $r_1$. For each $i\in\{\overline1,\cdots,\overline m,1,\cdots,n\}$, denote by $l(i)$ the length of the row where the $i${th} box lies in. It is obvious that the pair $i,j\in\{\overline1,\cdots,\overline m,1,\cdots,n\}$ with $\col(j)-\col(i)=-1$ occurs only when $l(i)$ and $l(j)$ have the different parity.

Let $i\in\{\overline1,\cdots,\overline m,1,\cdots,n\}$ be the smallest number among all the numbering boxes which lie in the rows of lengths sharing the different parity from that of $r_1$. Consider the set $K$ consisting of the $k${th} boxes with $k\in\{\overline1,\cdots,\overline m,1,\cdots,n\}$ such that $\row(k)<\row(i)$. In the Dynkin pyramid $P_{r,q}$, pick the $j${th} box from $K$ satisfying  $\col(j)-\col(i)=-1$, which  
implies that both the $j$th and  the $(j+1)${th} boxes  lie in the same row,   below left and  below  right of the $i$th box respectively. Obviously, the elements $e_{i,j},e_{j+1,i}\in\ggg^{P_{r,q}}(-1)$ by the definition. Carrying left translation for the $\row(i)${th} row by one unit from $P_{r,q}$, it follows that $\col'(j)-\col'(i)=0$ (i.e. $\text{deg}'e_{i,j}=0$) and $\col'(i)-\col'(j+1)=-2$ (i.e. $\text{deg}'e_{j+1,i}=-2$) in the shifted pyramid $P'$. Then $e_{i,j}$ falls in $\mathfrak{l}_1$, and $e_{j+1,i}$ in $\mathfrak{l}_2$ by the definition.

Recall the ordered index set $I=\{\overline1<\cdots<\overline m<1<\cdots<n\}$. Generally, for any  box of number $k\geqslant i$, by the same arguments as above one can find all pairs $(k,l)$ with $l\leqslant k$ such that  $\text{deg}\,e_{k,l}=-1$ as well as $\text{deg}'e_{k,l}=0$ and $\text{deg}'e_{l+1,k}=-2$. This implies that  $e_{k,l},e_{l+1,k}\in\ggg^{P_{r,q}}(-1)$, with $e_{k,l}\in\ggg^{P'}(0)\subseteq\mathfrak{p}'$ and $e_{l+1,k}\in\ggg^{P'}(-2)$.  When $k$ runs over all the $j$'s such that $j\geqslant i$, we can see that  $\mathfrak{l}_1$ is exactly spanned by those $e_{k,l}$'s, and $\mathfrak{l}_2$ spanned by those  $e_{l+1,k}$'s. 
Consequently, all the discussion above and the claim in Step 1 show that $\ppp'\supseteq \ppp\oplus\frak{l}_1$, $\ppp'\cap\mathfrak{l}_2=\{0\}$, and
\begin{equation}\label{half}
\underline{\dim}\,\mathfrak{l}_1=\underline{\dim}\,\mathfrak{l}_2=\frac{\underline{\dim}\,\ggg^{P_{r,q}}(-1)}{2}. \end{equation}
On the other hand, Lemma \ref{eo} shows that $\text{deg}'e_{i,j}=\text{deg}\,e_{i,j}$ or $\text{deg}'e_{i,j}=\text{deg}\,e_{i,j}\pm1$ for the pairs $i, j\in\{\overline1,\cdots,\overline m,1,\cdots,n\}$. Hence $\ppp'\subseteq \ppp\oplus\ggg^{P_{r,q}}(-1)$. Summing up, we must have $\ppp'=\ppp\oplus\mathfrak{l}_1$, completing the proof of the claim in Step 2.

Combining \eqref{half} and the claim in Step 2, we complete the proof.
\end{proof}
\begin{example}
Retain the notations in Example~\ref{expyr1}. In the Dynkin Pyramid $P_{r,q}$, one has
$\mathfrak{p}=\mathfrak{p}_{\bar0}\oplus\mathfrak{p}_{\bar1}$ with
\begin{equation*}
\begin{array}{lll}
\mathfrak{p}_{\bar0}&=&\bbk e_{\overline1,\overline1}+{\bbk}e_{\overline1,\overline2}+{\bbk}e_{\overline1,\overline3}+{\bbk}e_{\overline1,\overline4}
+{\bbk}e_{\overline2,\overline2}+{\bbk}e_{\overline2,\overline3}+{\bbk}e_{\overline2,\overline4}+{\bbk}e_{\overline3,\overline3}
+{\bbk}e_{1,1}+\\
&&{\bbk}e_{1,2}+{\bbk}e_{1,3}+{\bbk}e_{2,2}+{\bbk}e_{3,2}+{\bbk}e_{3,3}+{\bbk}e_{\overline4,\overline2}+{\bbk}e_{\overline4,\overline3}
+{\bbk}e_{\overline4,\overline4},
\end{array}
\end{equation*}and
\begin{equation*}
\begin{array}{lll}
\mathfrak{p}_{\bar1}&=&{\bbk}e_{\overline1,1}+{\bbk}e_{\overline1,2}+{\bbk}e_{\overline1,3}
+{\bbk}e_{\overline2,2}+{\bbk}e_{\overline2,3}+{\bbk}e_{1,\overline2}+{\bbk}e_{1,\overline3}+{\bbk}
e_{1,\overline4}+{\bbk}e_{2,\overline3}+\\&&{\bbk}e_{3,\overline2}+{\bbk}e_{3,\overline3}+{\bbk}e_{3,\overline4}+{\bbk}
e_{\overline4,2}+{\bbk}e_{\overline4,3}.
\end{array}
\end{equation*}
Moreover,
\begin{equation*}
\begin{array}{lll}
\ggg^{P_{r,q}}(-1)_{\bar0}&=&{\bbk}e_{2,3}+{\bbk}e_{3,1};\\
\ggg^{P_{r,q}}(-1)_{\bar1}&=&{\bbk}e_{\overline2,1}+{\bbk}e_{\overline3,2}+{\bbk}e_{1,\overline1}+{\bbk}e_{2,\overline2}+
{\bbk}e_{2,\overline4}+{\bbk}e_{\overline4,1}.
\end{array}
\end{equation*}
One can easily calculate that $\underline{\dim}\,\mathfrak{p}=(17,14)$ and $\underline{\dim}\,\ggg^{P_{r,q}}(-1)=(2,6)$.

In the shifted pyramid $P'$, one has
$\mathfrak{p}'=\mathfrak{p}'_{\bar0}\oplus\mathfrak{p}'_{\bar1}$ with
\begin{equation*}
\begin{array}{lll}
\mathfrak{p}'_{\bar0}&=&{\bbk}e_{\overline1,\overline1}+{\bbk}e_{\overline1,\overline2}+{\bbk}e_{\overline1,\overline3}+{\bbk}
e_{\overline1,\overline4}+{\bbk}e_{\overline2,\overline2}+{\bbk}e_{\overline2,\overline3}+{\bbk}e_{\overline2,\overline4}+{\bbk}e_{\overline3,\overline3}
+{\bbk}e_{1,1}+\\&&{\bbk}e_{1,2}+{\bbk}e_{1,3}+{\bbk}e_{2,2}+
{\bbk}e_{2,3}+{\bbk}e_{3,2}+{\bbk}e_{3,3}+{\bbk}e_{\overline4,\overline2}+{\bbk}e_{\overline4,\overline3}+{\bbk}e_{\overline4,\overline4},
\end{array}
\end{equation*}and
\begin{equation*}
\begin{array}{lll}
\mathfrak{p}'_{\bar1}&=&{\bbk}e_{\overline1,1}+{\bbk}e_{\overline1,2}+{\bbk}e_{\overline1,3}
+{\bbk}e_{\overline2,2}+{\bbk}e_{\overline2,3}+{\bbk}e_{1,\overline1}+{\bbk}e_{1,\overline2}+{\bbk}e_{1,\overline3}
+{\bbk}e_{1,\overline4}+\\
&&{\bbk}e_{2,\overline2}+{\bbk}e_{2,\overline3}+{\bbk}e_{2,\overline4}+{\bbk}e_{3,\overline2}+{\bbk}e_{3,\overline3}+{\bbk}e_{3,\overline4}
+{\bbk}e_{\overline4,2}+{\bbk}e_{\overline4,3},
\end{array}
\end{equation*}
thus $$\underline{\dim}\,\mathfrak{p}'=\underline{\dim}\,\mathfrak{p}+\frac{\underline{\dim}\,\ggg^{P_{r,q}}(-1)}{2}=(18,17).$$
\end{example}

\begin{rem} With the restrictions that $p\nmid(m-n)$ and $p>2$ for the field $\mathds{k}=\overline{\mathbb{F}}_p$, all the arguments and conclusions in this section are still valid for $\mathfrak{sl}_{m|n}$, with a few modifications in the proof.
\end{rem}
\begin{rem}
We call a $\mathbb{Z}$-grading $\mathfrak{g}=\bigoplus_{i\in\mathbb{Z}}\mathfrak{g}(i)$ of $\mathfrak{g}$ good if it admits an element $e\in\mathfrak{g}(2)_{\bar0}$ satisfying the conditions (1) and (2) in Proposition \ref{Dynkingradingprop}.
Let $\ggg=\mathfrak{gl}_{m|n}$, and $(r,q)$ be a partition of $(m|n)$. Over the field of complex numbers, Hoyt showed that if $P$ is a pyramid from $Pyr(r,q)$ as defined in \S\ref{ordinarypyramid}, then the grading $\ggg=\bigoplus_{i\in\mathbb{Z}}\ggg^P(i)$ associated to $P$ is good (c.f. \cite[Theorem 7.2]{H}). Recall \cite[Corollary 4.7]{H} shows that any good grading of $\ggg$ coincides with the eigenspace decomposition of ad\,$H$ for some semisimple element $H\in\ggg$. By the same discussion as \cite[\S2.5]{W} for the Lie algebra case, we can further show that Proposition \ref{Dynkingradingprop}(3)-(6) hold for any good gradings of $\ggg$ over $\bbc$.

In fact, one can check that Proposition \ref{Dynkingradingprop} is still valid for the good gradings of $\mathfrak{gl}_{m|n}$ over the algebraically closed field $\mathds{k}=\overline{\mathbb{F}}_p$ in characteristic $p\gg 0$, and also \eqref{dim}. Then a parabolic subalgebra $\mathfrak{p}'$ of $\ggg$ satisfying \eqref{keyequation} can be formulated directly in this case, bypassing the construction of shifted pyramids. We can proceed as follows.

For each nilpotent matrix of Jordan canonical form $e=e_{r,q}$ of type $(r,q)$, one can construct a pyramid $P_{Y}$ to be the Young diagram, number the boxes in it by the same way as that in \S\ref{ordinarypyramid}, and define the corresponding grading of $\ggg$. This grading is good, and we call it the Young grading. It is obvious that the Young grading is even by the construction. Let $\mathfrak{p}_{Y}=\sum\limits_{\begin{subarray}{c}i,j\in\{\overline1,\cdots,\overline m,1,\cdots,n\},\\ \col(i)\leqslant \col(j)\end{subarray}}\bbk e_{i,j}$ be a parabolic subalgebra associated to $P_{Y}$. Since \eqref{dim} holds for the Young grading, direct computation shows that the restricted subalgebra $\mathfrak{p}_{Y}$ of $\ggg$ satisfies the equation \eqref{keyequation}.
\end{rem}

\section{On the existence of KW modules for $\mathfrak{gl}_{m|n}$ and $\mathfrak{sl}_{m|n}$}\label{3}
This section is the main part of the paper. In virtue of the consequences on the shifted pyramids in \S\ref{2}, we
will show that the lower bounds in the super KW property (see Proposition \ref{wzd} and Proposition \ref{wzd2}) for $\ggg=\mathfrak{gl}_{m|n}$ and $\mathfrak{sl}_{m|n}$ are accessible.

\subsection{The nilpotent $p$-character case}
We will firstly deal with the nilpotent case. In fact, we have
\begin{prop}\label{main2}
Let $\ggg=\mathfrak{gl}_{m|n}$ or $\mathfrak{sl}_{m|n}$ with the restriction that $p>2$, and additionally $p\nmid(m-n)$ for $\ggg=\mathfrak{sl}_{m|n}$. Fix a nilpotent $p$-character $\chi\in\ggg^*_{\bar0}$. Then the reduced enveloping algebra $U_\chi(\ggg)$  admits Kac-Weisfeiler modules, i.e. the irreducible modules of  dimension $p^{\frac{d_0}{2}}2^{\frac{d_1}{2}}$.
\end{prop}

\begin{proof} Since the non-degenerate supersymmetric bilinear form on $\ggg$ enables us to identify nilpotent $p$-characters of $\ggg$ with nilpotent elements in $\ggg_{\bar0}$, we can assume that $\chi$ corresponds to the nilpotent element $e\in\ggg_{\bar0}$, i.e. $\chi(\cdot)=(e,\cdot)$. Recall that $(\mbox{GL}_{m|n})_\ev=\mbox{GL}(m,\bbk)\times \mbox{GL}(n,\bbk)$, and denote by $(\mbox{SL}_{m|n})_\ev:=\mbox{SL}(m,\bbk)\times \mbox{SL}(n,\bbk)$. Up to $(\mbox{GL}_{m|n})_\ev$-conjugation (resp. $(\mbox{SL}_{m|n})_\ev$-conjugation), we might as well assume that $e=e_{r,q}$ is a nilpotent matrix of Jordan canonical form of type $(r,q)$, where $(r,q)$ denotes the pair of partitions $r\vdash m$ and $q\vdash n$.

Associated with $e$, we can define the Dynkin pyramid $P_{r,q}$ and the shifted one $P'$ respectively, by the same way as that in  \S\ref{2}. As defined in \eqref{pp'2}, let $\ppp$ and $\ppp'$ be the parabolic subalgebras of $\ggg$ corresponding to the pyramids $P_{r,q}$ and $P'$, respectively. Retain the notations as those in \S\ref{2}. Since $\chi(y)=0$ for all $y\in\ggg^{P'}(i)$ with $i\neq -2$, and $\mathfrak{p}'\subseteq\bigoplus_{i\geqslant0}\ggg^{P'}(i)$, we have $\chi(\mathfrak{p}')=0$. As $\mathfrak{p}'$ is a restricted subalgebra of $\ggg$, it is immediate that the $\mathds{k}$-algebra $U_\chi(\mathfrak{p}')$ admits a one-dimensional restricted representation, and we denote it by $V$.  Recall Proposition~\ref{pp'} shows that
\begin{equation*}
\underline{\dim}\,\mathfrak{p}'=\underline{\dim}\,\mathfrak{p}+\frac{\underline{\dim}\,\ggg^{P_{r,q}}(-1)}{2}=\sum_{i\geqslant 0}\underline{\dim}\,\ggg^{P_{r,q}}(i)+\frac{\underline{\dim}\,\ggg^{P_{r,q}}(-1)}{2},
\end{equation*}
then we have $\text{dim}\,U_\chi(\mathfrak{p}')=p^{\frac{\text{dim}\,\ggg^{P_{r,q}}(-1)_{\bar0}}{2}+\sum\limits_{i\geqslant 0}\text{dim}\,\ggg^{P_{r,q}}(i)_{\bar0}}2^{\frac{\text{dim}\,\ggg^{P_{r,q}}(-1)_{\bar1}}{2}+\sum\limits_{i\geqslant 0}\text{dim}\,\ggg^{P_{r,q}}(i)_{\bar1}}$. Recall $d_i=\text{dim}\,\ggg_i-\text{dim}\,\ggg^e_i$ for $i\in\mathbb{Z}_2$, and $\text{dim}\,U_\chi(\ggg)=p^{\text{dim}\,\ggg_{\bar{0}}}2^{\text{dim}\,\ggg_{\bar{1}}}$ by the definition. Combining with \eqref{dim} it follows that the induced $U_\chi(\ggg)$-module $U_\chi(\ggg)\otimes_{U_\chi(\mathfrak{p}')}V$ has dimension
\begin{equation}
\begin{array}{lll}
\frac{\text{dim}\,U_\chi(\ggg)}{\text{dim}\,U_\chi(\mathfrak{p}')}&=&
\frac{p^{\sum\limits_{i\in\mathbb{Z}}\text{dim}\,\ggg^{P_{r,q}}(i)_{\bar0}}2^{\sum\limits_{i\in\mathbb{Z}}\text{dim}\,\ggg^{P_{r,q}}(i)_{\bar1}}}
{p^{\frac{\text{dim}\,\ggg^{P_{r,q}}(-1)_{\bar0}}{2}+\sum\limits_{i\geqslant 0}\text{dim}\,\ggg^{P_{r,q}}(i)_{\bar0}}2^{\frac{\text{dim}\,\ggg^{P_{r,q}}(-1)_{\bar1}}{2}+\sum\limits_{i\geqslant 0}\text{dim}\,\ggg^{P_{r,q}}(i)_{\bar1}}}\\
&=&p^{\frac{\text{dim}\,\ggg^{P_{r,q}}(-1)_{\bar0}}{2}+\sum\limits_{i\leqslant -2}\text{dim}\,\ggg^{P_{r,q}}(i)_{\bar0}}2^{\frac{\text{dim}\,\ggg^{P_{r,q}}(-1)_{\bar1}}{2}+\sum\limits_{i\leqslant -2}\text{dim}\,\ggg^{P_{r,q}}(i)_{\bar1}}\\
&=&p^{\frac{d_0}{2}}2^{\frac{d_1}{2}},
\end{array}
\end{equation} thereby irreducible by Proposition \ref{wzd}.
We complete the proof.
\end{proof}

\subsection{Arbitrary $p$-character case} We will extend the above arguments to the arbitrary $p$-character case in the  sequent.

First recall some basics on simple superalgebras (cf. \cite[\S12.1]{KL}).
Let $V$ be a superspace with $\text{\underline{dim}}\,V=(m,n)$, then $\mathcal{M}(V):=\text{End}_\mathds{k}(V)$ is a superalgebra with $\text{\underline{dim}}\,\mathcal{M}(V)=(m^2+n^2,2mn)$. The algebra $\mathcal{M}(V)$ is defined uniquely up to an isomorphism by the superdimension $(m,n)$ of $V$. Thus we can speak of the superalgebra $\mathcal{M}_{m,n}:=\mathcal{M}(V)$. We have an isomorphism of $\mathds{k}$-algebras
\begin{equation}\label{MM}
\mathcal{M}_{m,n}\otimes\mathcal{M}_{k,l}\cong\mathcal{M}_{mk+nl,ml+nk}.
\end{equation}
Moreover,  \cite[Example 12.1.1]{KL} shows that $\mathcal{M}_{m,n}$ is a simple superalgebra.

Now we will recall some known facts on the representations of finite-dimensional superalgebras \cite[\S12]{KL}. Let $A$ and $B$ be finite-dimensional superalgebras over $\mathds{k}$. Given left modules $V$ and $W$ over $\mathds{k}$-algebras $A$ and $B$ respectively, the (outer) tensor product $V\boxtimes W$ is the space $V\otimes W$ considered as an $A\otimes B$-module via
$$(a\otimes b)(v\otimes w)=(-1)^{|b||v|}av\otimes bw\qquad(a\in A,\,b\in B,\,v\in V,\,w\in W).$$

Return to our concern with $\ggg=\mathfrak{gl}_{m|n}$ or $\mathfrak{sl}_{m|n}$. Let $\xi=\xi_{s}+\xi_{n}$ be the Jordan decomposition of $\xi\in\ggg^*_{\bar0}$ (we regard $\xi\in\ggg^*$ by letting $\xi(\ggg_{\bar1})=0$). Under the isomorphism $\ggg^*_{\bar0}\cong\ggg_{\bar0}$ induced by the non-degenerate bilinear form $(\cdot,\cdot)$ on $\ggg_{\bar0}$, we identify $\xi$ with $x\in \ggg_\bz$, which has the Jordan decomposition $x=s+n$ on $\ggg_{\bar0}$.
Without loss of generality, we might as well assume that $s$ falls in the canonical Cartan subalgebra $\frak{h}$ of $\ggg_\bz$.
Denote by $\mathfrak{l}:=\ggg^{s}$ the centralizer of $s$ in $\ggg$. Let $\Phi$ be the canonical root system of $\ggg$, and $\Phi(\mathfrak{l}):=\{\alpha\in\Phi~|~\alpha(s)=0\}$. By \cite[Proposition 5.1]{WZ}, $\mathfrak{l}$ is always a direct sum of basic Lie superalgebras with a system $\Pi$ of simple roots of $\ggg$ such that $\Pi\cap \Phi(\mathfrak{l})$ is a system of simple roots of $\Phi(\mathfrak{l})$ (note that a toral subalgebra of $\ggg$ may also appear in the summand).

Set $$\mathfrak{l}=\ggg^{s}=\bigoplus\limits_{i=1}^r(\ggg)_i\oplus\mathfrak{t},$$ where $(\ggg)_i$ is a basic Lie superalgebra for $1\leqslant i\leqslant r$, and
$\mathfrak{t}$ is a toral subalgebra of $\ggg$. As $\ggg=\mathfrak{gl}_{m|n}$ or $\mathfrak{sl}_{m|n}$, by easy calculation one can conclude that each summand of $\bigoplus\limits_{i=1}^r(\ggg)_i$ is always isomorphic to a Lie superalgebra of type $\mathfrak{gl}_{M|N}$ or  $\mathfrak{sl}_{M|N}$ for some $M,N\in\mathbb{Z}_+$. It is notable that $\xi_s|_{\mathfrak{l}}=0$, hence $\xi|_{\mathfrak{l}}=\xi_n|_{\mathfrak{l}}$ is a nilpotent $p$-character of $\mathfrak{l}$. Set $\xi_{n}=\xi_1+\cdots+\xi_r$, where $\xi_i\in(\ggg)_i^*$ (which can be viewed in $\mathfrak{l}^*$ by letting $\xi_i(y)=0$ for all $y\in\bigoplus\limits_{j\neq i}(\ggg)_j\oplus\mathfrak{t}$) for $1\leqslant i\leqslant r$. Let $n=n_1+\cdots+n_r$ be the corresponding decomposition of $n$ in $\mathfrak{l}$ such that $\xi_i(\cdot)=(n_i,\cdot)$ for $1\leqslant i\leqslant r$.

Since $\mathfrak{l}=\bigoplus\limits_{i=1}^r(\ggg)_i\oplus\mathfrak{t}$, we have
\begin{equation}\label{lsumt'}
U_{\xi_{n}}(\mathfrak{l})\cong U_{\xi_{n}}(\bigoplus\limits_{i=1}^r(\ggg)_i\oplus\mathfrak{t})\cong
U_{\xi_{n}}(\bigoplus\limits_{i=1}^r(\ggg)_i)\otimes U_0(\mathfrak{t})
\end{equation} as $\mathds{k}$-algebras. Set
\begin{equation}\label{arbitdim}
\begin{array}{rllrlll}
d_0&:=&\text{dim}\,\ggg_{\bar0}-\text{dim}\,\ggg^{x}_{\bar0},&
d_1&:=&\text{dim}\,\ggg_{\bar1}-\text{dim}\,\ggg^{x}_{\bar1},\\
(d_0)_i&:=&\text{dim}\,((\ggg)_i)_{\bar0}-\text{dim}\,((\ggg)^{n_i}_i)_{\bar0},&
(d_1)_i&:=&\text{dim}\,((\ggg)_i)_{\bar1}-\text{dim}\,((\ggg)^{n_i}_i)_{\bar1},
\end{array}
\end{equation}where $\ggg^{x}$ denotes the centralizer of $x$ in $\ggg$, and $(\ggg)^{n_i}_i$ the centralizer of $n_i$ in $(\ggg)_i$ for $i\in\{1,\cdots,r\}$. Define
\begin{equation}\label{sumdim}
d'_0:=\sum\limits_{i=1}^r(d_0)_i,\qquad
d'_1:=\sum\limits_{i=1}^r(d_1)_i.
\end{equation}

Retain the notations as above. We now claim that
\begin{lemma}\label{d'}
The reduced enveloping algebra $U_{\xi_{n}}(\mathfrak{l})$ over $\mathds{k}=\overline{\mathbb{F}}_p$ admits a representation of dimension $p^{\frac{d'_0}{2}}2^{\frac{d'_1}{2}}$.
\end{lemma}

\begin{proof}
(i) First consider the representations of the $\mathds{k}$-algebra $U_{\xi_{n}}(\bigoplus\limits_{i=1}^r(\ggg)_i)$. Since $\xi_n|_{\mathfrak{l}}$ is nilpotent and each $(\ggg)_i$ ($1\leqslant i\leqslant r$) is always isomorphic to a Lie superalgebra of type $\mathfrak{gl}_{M|N}$ or  $\mathfrak{sl}_{M|N}$ for some $M,N\in\mathbb{Z}_+$, Proposition \ref{main2} shows that the $\mathds{k}$-algebra $U_{\chi_i}((\ggg)_i)$ admits an irreducible representation of dimension $p^{\frac{(d_0)_i}{2}}2^{\frac{(d_1)_i}{2}}$ for each $1\leqslant i\leqslant r$; call it $V_i$. By induction one can show that $V_1\boxtimes V_2\boxtimes\cdots\boxtimes V_r$ is a $U_{\xi_{n}}(\bigoplus\limits_{i=1}^r(\ggg)_i)\cong\bigotimes\limits_{i=1}^r U_{\xi_i}((\ggg)_i)$-module of dimension $p^{\sum\limits_{i=1}^r\frac{(d_0)_i}{2}}2^{\sum\limits_{i=1}^r\frac{(d_1)_i}{2}}=p^{\frac{d'_0}{2}}2^{\frac{d'_1}{2}}$.

(ii) Now discuss the representations of the $\mathds{k}$-algebra $U_0(\mathfrak{t})$. As $\mathfrak{t}$ is a toral subalgebra of $\ggg$ with a basis $\{t_1,\cdots,t_d\}$ such that $t_i^{[p]}=t_i$ for all $1\leqslant i\leqslant d$, then $U_0(\mathfrak{t})\cong A_1^{\otimes d}$ where $A_1\cong\mathds{k}[X]/(X^p-X)$ is a $p$-dimensional commutative semisimple $\mathds{k}$-algebra (in the usual sense, not super) whose irreducible representations are all $1$-dimensional. Hence the $\mathds{k}$-algebra $U_0(\mathfrak{t})$ admits an irreducible representation of dimension $1$; call it $W$.

As $U_{\xi_{n}}(\mathfrak{l})\cong U_{\xi_{n}}(\bigoplus\limits_{i=1}^r(\ggg)_i)\otimes U_0(\mathfrak{t})$ by \eqref{lsumt'}, all the discussions above show that $(\boxtimes_{i=1}^rV_i)\boxtimes W$ is a $U_{\xi_{n}}(\mathfrak{l})$-module of dimension $p^{\frac{d'_0}{2}}2^{\frac{d'_1}{2}}$.
\end{proof}

\subsection{Main result} Let $\mathfrak{b}=\mathfrak{h}\oplus\mathfrak{n}$ be the Borel subalgebra associated to $\Pi$. Define a parabolic subalgebra $\mathfrak{p}_\Pi=\mathfrak{l}+\mathfrak{b}=\mathfrak{l}\oplus\mathfrak{u}$, where $\mathfrak{u}$ is the nilradical of $\mathfrak{p}_\Pi$.
Let $\mathfrak{u}^-$ be the complement space of $\mathfrak{p}_\Pi$ in $\ggg$ such that  $\ggg=\mathfrak{p}_\Pi\oplus\mathfrak{u}^-=\mathfrak{u}\oplus\mathfrak{l}\oplus\mathfrak{u}^-$  as vector spaces.
Since $\xi(\mathfrak{u})=0$ and $\xi|_{\mathfrak{l}}=\xi_{n}|_{\mathfrak{l}}$ is nilpotent by \cite[\S5.1]{WZ}, any $U_\xi(\mathfrak{l})$-mod can be regarded as a $U_\xi(\mathfrak{p}_\Pi)$-mod with a trivial action of $\mathfrak{u}$. In \cite[Theorem 5.2]{WZ} Wang-Zhao proved that the $\mathds{k}$-algebras $U_\xi(\ggg)$ and $U_\xi(\mathfrak{l})$ are Morita equivalent, and any irreducible $U_\xi(\ggg)$-module can be induced from an irreducible $U_\xi(\mathfrak{l})$-mod (considered as a $U_\xi(\mathfrak{p}_\Pi)$-mod with a trivial action of $\mathfrak{u}$) by
\begin{equation}\label{gp}
U_\xi(\ggg)\otimes_{U_\xi(\mathfrak{p}_\Pi)}-:\qquad U_\xi(\mathfrak{l})\text{-}mod\rightarrow U_\xi(\ggg)\text{-}mod.
\end{equation}

Now we are in a position to verify the existence of Kac-Weisfeier modules in the $U_\xi(\ggg)$-module category with arbitrary $p$-character $\xi$.

\begin{theorem}\label{main4}
Let $\ggg=\mathfrak{gl}_{m|n}$ or $\mathfrak{sl}_{m|n}$, and let $\xi\in \ggg^*_{\bar0}$.  Retain the notations as \eqref{arbitdim} and \eqref{sumdim} with additional assumption  that $p\nmid(m-n)$ for $\ggg=\mathfrak{sl}_{m|n}$. Then the reduced enveloping algebra $U_\xi(\ggg)$ admits a Kac-Weisfeiler module, i.e. an irreducible module of dimension $p^{\frac{d_0}{2}}2^{\frac{d_1}{2}}$.
\end{theorem}
\begin{proof}
First note that \cite[Theorem 5.6]{WZ} shows
\begin{equation}\label{gtol}
\begin{split}
\text{\underline{dim}}\,\ggg-\text{\underline{dim}}\,\ggg^{x}&=\text{\underline{dim}}\,
\ggg-\text{\underline{dim}}\,\mathfrak{l}^{n}\\
&=2\text{\underline{dim}}\,\mathfrak{u}^-+(\text{\underline{dim}}\,\mathfrak{l}-\text{\underline{dim}}\,\mathfrak{l}^{n}),
\end{split}
\end{equation} where $\mathfrak{l}^{n}$ denotes the centralizer of $n$ in $\mathfrak{l}$.
Since $\mathfrak{l}=\bigoplus\limits_{i=1}^r(\ggg)_i\oplus\mathfrak{t}$ and $n\in \bigoplus\limits_{i=1}^r(\ggg)_i$, it is obvious that the centralizer of $n$ in $\mathfrak{t}$ satisfies $\mathfrak{t}^{n}=\mathfrak{t}$, then
\begin{equation}\label{ltosum}
\begin{split}
\text{\underline{dim}}\,\mathfrak{l}-\text{\underline{dim}}\,\mathfrak{l}^{n}=&\sum\limits_{i=1}^r
(\text{\underline{dim}}\,(\ggg)_i-\text{\underline{dim}}\,(\ggg)_i^{n_i})+
\text{\underline{dim}}\,\mathfrak{t}-\text{\underline{dim}}\,\mathfrak{t}^{ n}\\=&(\sum\limits_{i=1}^r(d_0)_i,\sum\limits_{i=1}^r(d_1)_i).\end{split}
\end{equation} As $\text{\underline{dim}}\,\mathfrak{u}^-=\text{\underline{dim}}\,\mathfrak{u}$, \eqref{gtol} shows that
\begin{equation}\label{dimu}
\begin{array}{rclllll}
\text{dim}\,\mathfrak{u}^-_{\bar0}&=&\text{dim}\,\mathfrak{u}_{\bar0}&=&\frac{d_0-\sum\limits_{i=1}^r(d_0)_i}{2}&=
&\frac{d_0-d'_0}{2};\\ \text{dim}\,\mathfrak{u}^-_{\bar1}&=&\text{dim}\,\mathfrak{u}_{\bar1}&=&\frac{d_1-\sum\limits_{i=1}^r(d_1)_i}{2}&=
&\frac{d_1-d'_1}{2}.
\end{array}
\end{equation}
Recall Lemma~\ref{d'} shows that the $\mathds{k}$-algebra $U_{\xi_{n}}(\mathfrak{l})$ admits a representation of dimension $p^{\frac{d'_0}{2}}2^{\frac{d'_1}{2}}$; call it $V'$. Then $V:=U_\xi(\ggg)\otimes_{U_\xi(\mathfrak{p}_\Pi)}V'$ is a $U_\xi(\ggg)$-module of dimension $$p^{\frac{d'_0}{2}}2^{\frac{d'_1}{2}}\cdot p^{\frac{d_0-d'_0}{2}}2^{\frac{d_0-d'_1}{2}}=
p^{\frac{d_0}{2}}2^{\frac{d_1}{2}}$$ by \eqref{gp}, thereby irreducible by Proposition \ref{wzd2}.
Such a $V$ is a Kac-Weisfeiler module. We complete the proof.
\end{proof}

\subsection*{Acknowledgements} The authors would like to express their deep thanks to Weiqiang Wang for  helpful discussions.

\end{document}